\documentclass[12pt]{amsart}

\usepackage[margin=1.5in]{geometry}

\usepackage{amsmath}
\usepackage{amsfonts}
\usepackage{amssymb}
\usepackage{longtable}
\usepackage{amsthm}
\usepackage{enumerate}
\usepackage{graphicx}
\usepackage{hyperref}
\usepackage{enumitem}

\input xy
\xyoption{all}

\theoremstyle{plain}
\newtheorem{theorem}{Theorem}[section]

\newtheorem{lemma}[theorem]{Lemma}
\newtheorem{proposition}[theorem]{Proposition}
\theoremstyle{definition}

\newtheorem{rem}[theorem]{Remark}

\DeclareMathOperator{\aut}{Aut}

\newcommand{\bb}[1]{\mathbb{#1}}
\newcommand{\stab}{\mathrm{Stab}}
\newcommand{\im}{\mathrm{im}}

\def\Z{{\mathbb{Z}}}

\def\R{{\mathbb{R}}}

\def\C{{\mathbb{C}}}

\makeatletter
\def\ps@pprintTitle{%
  \let\@oddhead\@empty
  \let\@evenhead\@empty
  \let\@oddfoot\@empty
  \let\@evenfoot\@oddfoot
}
\makeatother

\usepackage{marvosym}

\title{Smooth quotients of abelian surfaces by finite groups that fix the origin}
\author{Robert Auffarth}
\address{R. Auffarth \\Departamento de Matem\'aticas, Facultad de
Ciencias, Universidad de Chile\\ Las Palmeras 3425, \~Nu\~noa, Santiago, Chile}
\email{rfauffar@uchile.cl}
\author{Giancarlo Lucchini Arteche}
\address{G. Lucchini Arteche \\Departamento de Matem\'aticas, Facultad de
Ciencias, Universidad de Chile\\ Las Palmeras 3425, \~Nu\~noa, Santiago, Chile}
\email{giancarlo.lucchini-arteche@polytechnique.edu}
\author{Pablo Quezada}
\address{P. Quezada \\Facultad de Matem\'aticas, Pontificia Universidad Cat\'olica\\ Vicu\~na Mackenna 4860, Macul, Santiago, Chile}
\email{psquezada@uc.cl}

\thanks{The first and third authors were partially supported by CONICYT PIA ACT1415. The second author was partially supported by Fondecyt Grant 11170016 and PAI Grant 79170034.}


\begin{document}

\begin{abstract}
Let $A$ be an abelian surface and let $G$ be a finite group of automorphisms of $A$ fixing the origin. Assume that the analytic re\-pre\-sen\-ta\-tion of $G$ is irreducible. We give a classification of the pairs $(A,G)$ such that the quotient $A/G$ is smooth. In particular, we prove that $A=E^2$ with $E$ an elliptic curve and that $A/G\simeq\bb P^2$ in all cases. Moreover, for fixed $E$, there are only finitely many pairs $(E^2,G)$ up to isomorphism. This fills a small gap in the literature and
completes the classification of smooth quotients of abelian varieties by finite groups fixing the origin started by the first two authors.\\

\end{abstract}

\maketitle

\section{Introduction}

The purpose of this paper is to give a complete classification of all smooth quotients of abelian surfaces by finite groups that fix the origin, and is to be seen as the completion of the classification given in \cite{ALA} of smooth quotients of abelian varieties that fix the origin. This kind of quotients of abelian surfaces has already been studied by Tokunaga and Yoshida in \cite{TY}, where infinite 2-dimensional complex reflection groups, which are extensions of a finite complex reflection group $G$ by a $G$-invariant lattice, are classified. However, these do not cover all possible $G$-invariant lattices and hence not all possible group actions on abelian surfaces. Moreover, there seem to be some complex reflection groups that the authors missed, as can be seen by looking at Popov's classification of the same groups in \cite{Popov}.

The techniques used in this paper are similar, but not exactly the same, to the methods used in \cite{ALA}. Indeed, the ideas used in this last paper have been modified in order to apply them to the two-dimensional case. Moreover our approach is far different from that used in \cite{TY}. 

Our main theorem states the following:

\begin{theorem}\label{classification}
Let $A$ be an abelian surface and let $G$ be a (non-trivial) finite group of automorphisms of $A$ that fix the origin. Then the following conditions are equivalent:

\begin{itemize}
\item[(1)] $A/G$ is smooth and the analytic representation of $G$ is irreducible.
\item[(2)] $A/G\simeq\mathbb{P}^2$.
\item[(3)] There exists an elliptic curve $E$ such that $A\simeq E^2$ and $(A,G)$ satisfies exactly one of the following:
\begin{enumerate}[label=(\alph*)]
\item $G\simeq C^2\rtimes S_2$ where $C$ is a non-trivial (cyclic) subgroup of automorphisms of $E$ that fix the origin; here the action of $C^2$ is coordinatewise and $S_2$ permutes the coordinates.\label{ex1}
\item $G\simeq S_{3}$ and acts on 
\[A\simeq\{(x_1,x_2,x_3)\in E^{3}:x_1+x_2+x_3=0\},\] 
by permutations.\label{ex2}
\item $E=\C/\Z[i]$ and $G$ is the order 16 subgroup of $\mathrm{GL}_2(\Z[i])$ generated by:
\[\left\{\begin{pmatrix} -1 & 1+i \\ 0 & 1\end{pmatrix}\right.,\, \begin{pmatrix} -i & i-1 \\ 0 & i\end{pmatrix},\, \left.\begin{pmatrix} -1 & 0 \\ i-1 & 1\end{pmatrix} \right\},\]
acting on $A\simeq E^2$ in the obvious way. \label{ex3}
\end{enumerate}
\end{itemize}
\end{theorem}

The first two cases found in item $(3)$ of the above theorem were studied in detail in \cite{Auff} (in arbitrary dimension), where it was proven that both examples give the projective plane as a quotient. Throughout the paper we will refer to these two examples as Example \ref{ex1} and Example \ref{ex2}, respectively. The equivalent assertion for Example \ref{ex3} is Proposition \ref{prop ex3 smooth} in this paper. Note that, aside from Examples \ref{ex1} and \ref{ex2} which belong to infinite families, Example \ref{ex3} is the only new case of an action of $G$ on an abelian variety satisfying condition (1) from Theorem \ref{classification}, cf.~\cite[Thm.~1.1]{ALA}.

\begin{rem}
If $A/G$ is smooth and the analytic representation of $G$ is \emph{reducible}, then the results in \cite{ALA} imply that $A$ is isogenous to a product of two elliptic curves. The quotient is then either $\bb P^1\times\bb P^1$ (in which case $A=E_1\times E_2$) or a bielliptic surface.
\end{rem}

In \cite{Yoshi}, Yoshihara introduces the notion of a Galois embedding of a smooth projective variety. If $X$ is a smooth projective variety of dimension $n$ and $D$ is a very ample divisor that induces an embedding $X\hookrightarrow\mathbb{P}^N$, then the embedding is said to be \emph{Galois} if there exists an $(N-n-1)$-dimensional linear subspace $W$ of $\mathbb{P}^N$ such that $X\cap W=\varnothing$ and the restriction of the linear projection $\pi_W:\mathbb{P}^N\dashrightarrow\mathbb{P}^n$ to $X$ is Galois. Yoshihara specifically studies when abelian surfaces have a Galois embedding. He gives a classification of abelian surfaces having a Galois embedding, along with their Galois groups, and proves that after taking the quotient of the original abelian variety by the translations of the Galois group, the abelian variety must be isomorphic to the self-product of an elliptic curve. Unfortunately, his results were incomplete since they depended on a classification of smooth quotients like the one given in this paper, which Yoshihara attributed to Tokunaga and Yoshida in \cite{TY}. But as stated before, Tokunaga and Yoshida's results do not imply such a classification. Nevertheless, we can now safely say, thanks to Theorem \ref{classification}, that Yoshihara's results remain correct.\\

The structure of this paper is as follows: in Section \ref{sec group actions} we fix notations and give some preliminary results that will be needed in the proofs of Theorem \ref{classification}. The implication $(2)\Rightarrow (1)$ is obvious and $(3)\Rightarrow (2)$ was already treated in \cite{Auff} in the case of Examples \ref{ex1} and \ref{ex2}. Thus, we are mainly concerned with $(1)\Rightarrow (3)$, which we treat in Section \ref{sec reflectiongroup}. Finally, in Section \ref{sec ex3} we treat $(3)\Rightarrow(2)$ for Example \ref{ex3}, which is a construction of a different nature that only exists in the 2-dimensional case.

\section{Preliminaries on group actions on abelian varieties}\label{sec group actions}

We recall here some elementary results that were proved in \cite{ALA}. Let $A$ be an abelian surface and let $G$ be a group of automorphisms of $A$ that fix the origin, such that the quotient variety $A/G$ is smooth. By the Chevalley-Shephard-Todd Theorem, the stabilizer in $G$ of each point in $A$ must be generated by pseudoreflections; that is, elements that fix pointwise a divisor (i.e.~a curve) containing the point. In particular, $G=\stab_G(0)$ is generated by pseudoreflections and $G$ acts on the tangent space at the origin $T_0(A)$ (this is the analytic representation). In this context, a pseudoreflection is an element that fixes a line pointwise. We will often abuse notation and display $G$ as either acting on $A$ or $T_0(A)$; it will be clear from the context which action we are considering. 

In what follows, let $\mathcal{L}$ be a fixed $G$-invariant polarization on $A$ (take the pullback of an ample class on $A/G$, for example). For $\sigma$ a pseudoreflection in $G$ of order $r$, define
\begin{align*}
D_\sigma&:=\im(1+\sigma+\cdots+\sigma^{r-1}),\\
E_\sigma&:=\im(1-\sigma).
\end{align*}
These are both abelian subvarieties of $A$. The following result corresponds to \cite[Lem.~2.1]{ALA}.

\begin{lemma}\label{lemma Esigma Dsigma}
We have the following:
\begin{itemize}
\item[1.] $D_\sigma$ is the connected component of $\ker(1-\sigma)$ that contains 0 and $E_\sigma$ is the complementary abelian subvariety of $D_\sigma$ with respect to $\mathcal{L}$. In particular, $D_\sigma$ and $E_\sigma$ are elliptic curves.
\item[2.] $F_\sigma:=D_\sigma\cap E_\sigma$ consists of $2$-torsion points for $r=2,4$, of $3$-torsion points for $r=3$ and $D_\sigma\cap E_\sigma=0$ for $r=6$.
\end{itemize}
\end{lemma}

We will consider now a new abelian surface $B$ equipped with a $G$-equivariant isogeny to $A$, which we will call $G$-isogeny from now on. Let $\Lambda_A$ denote the lattice in $\C^2$ such that $A=\C^2/\Lambda_A$. Let $\Lambda_B\subseteq\Lambda_A$ be a $G$-invariant sublattice, and let $B:=C^2/\Lambda_B$ be the induced abelian surface, along with the $G$-isogeny
\[\pi:B\to A,\]
whose analytic representation is the identity. Note that this implies that $\sigma\in G$ is a pseudoreflection of $B$ if and only if it is a pseudoreflection of $A$. We may then consider the subvarieties $E_\sigma,D_\sigma,F_\sigma\subset A$ defined as above, which we will denote by $E_{\sigma,A},D_{\sigma,A}$ and $F_{\sigma,A}$. We do similarly for $B$. Note that, by definition, $\pi$ sends $E_{\sigma,B}$ to $E_{\sigma,A}$ and $D_{\sigma,B}$ to $D_{\sigma,A}$, hence $F_{\sigma,B}$ to $F_{\sigma,A}$. The following result was proved in \cite[Prop.~2.4]{ALA}

\begin{proposition}\label{prop Delta and D}
Let $\sigma\in G$ be a pseudoreflection and let $L$ be the line defining both $E_{\sigma,A}$ and $E_{\sigma,B}$. Assume that the map $F_{\sigma,B}\to F_{\sigma,A}$ is surjective and that $\Lambda_A\cap L=\Lambda_B\cap L$. Then $\ker(\pi)$ is contained in $D_{\tau\sigma\tau^{-1},B}$ for \emph{every} $\tau\in G$.\qed
\end{proposition}

Define $\Delta:=\ker(\pi)$. Since $\pi$ is $G$-equivariant, $G$ acts on $\Delta$ and hence we may consider the group $\Delta\rtimes G$. This group acts on $B$ in the obvious way: $\Delta$ acts by translations and $G$ by automorphisms. In particular, we see that the quotient $B/(\Delta\rtimes G)$ is isomorphic to $A/G$. We conclude this section by recalling a result on pseudoreflections in $\Delta\rtimes G$ (cf.~\cite[Lem.~2.5]{ALA}).

\begin{lemma}\label{main lemma}
Let $\sigma\in\Delta\rtimes G$ be a pseudoreflection. Then $\sigma=(t,\tau)$ with $\tau\in G$ a pseudoreflection and $t\in\Delta\cap E_{\tau,B}$.
\end{lemma}

\section{Proof of $(1)\Rightarrow (3)$}\label{sec reflectiongroup}
Assume (1), that is, we have an abelian surface $A$ with an action of a finite group $G$ that fixes the origin and such that $A/G$ is smooth and the analytic representation of $G$ is irreducible. Under these conditions, we see that $G$ is an irreducible finite complex reflection group in the sense of Shephard-Todd \cite{ST}. These groups were completely classified by Shephard and Todd in \cite{ST}. In the particular case of dimension 2, we get that $G$ is either one of 19 sporadic cases or it is isomorphic to a semidirect product $G(m,p):=H(m,p)\rtimes S_2$, where $p|m$, $m\geq 2$, and
\[H=H(m,p)=\{(\zeta_m^{a_1},\zeta_m^{a_2})\mid a_1+a_2\equiv 0\pmod p\}\subset \mu_m^2,\]
with $\zeta_m$ denoting a primitive $m$-th root of unity. The action of $S_2$ on $H$ is the obvious one. The case $G=G(2,2)$ is excluded since $G$ is then a Klein group and thus the representation is not irreducible. The action of $G$ on $\C^2$ is given as follows: $H$ acts on $\C^2$ coordinate-wise while $S_2$ premutes the coordinates.\\

Emulating the work done in \cite{ALA}, we wish to describe which of these actions actually appear on abelian surfaces and give smooth quotients. The sporadic cases were already treated in \cite{ALA} and were proven not to give a smooth quotient (cf.~\cite[\S3.3]{ALA}), so we may and will assume henceforth that $G=G(m,p)$ as above. This fixes a $G$-equivariant isomorphism of $T_0(A)$ with $\C^2$. We denote by $e_1$ and $e_2$ the canonical basis of $T_0(A)$ thus obtained.

\begin{lemma}\label{lem m leq 6}
Assume that $G$ acts on $A$ as above. Then $m\in\{2,3,4,6\}$.
\end{lemma}

\begin{proof}
Let $E$ be the image of $\mathbb{C}(e_1+e_2)$ in $A$ via the exponential map. We claim that it corresponds to an elliptic curve. Indeed, consider the elements $\tau_1=(\zeta_m,\zeta_m^{-1}),\tau_2=(\zeta_m^{-1},\zeta_m)\in H$. A direct computation shows that, for $\sigma=(1\, 2)\in S_n\subset G$, $\im(\tau_1+\tau_2+\tau_1\sigma+\tau_2\sigma)=\C(e_1+e_2)$ and hence $E=(\tau_1+\tau_2)(1+\sigma)(A)$ is an elliptic curve.

We see then that the element $\tau_1+\tau_2$ induces a \emph{real} endomorphism of $E$ corresponding to multiplication by $\zeta_m+\zeta_m^{-1}\in\R$. Since real automorphisms of elliptic curves correspond to multiplication by integers, we see that $m\in\{2,3,4,6\}$ since $\zeta_m+\zeta_m^{-1}\in\Z$ if and only if $m\in\{1,2,3,4,6\}$ and here $m\geq 2$.
\end{proof}

Having proved this result, we see that there is a finite list of cases to be analyzed, that is:
\[(m,p)\in\{(2,1),(2,2),(3,1),(4,1),(6,1),(3,3),(4,2),(4,4),(6,2),(6,3),(6,6)\}.\]
Recall that we have already eliminated the case $(2,2)$ since the analytic representation of $G(2,2)$ is not irreducible. Moreover, it is well-known that there is an exceptional isomorphism of complex reflection groups between $G(4,4)$ and $G(2,1)$. We will prove then the following:
\begin{itemize}
\item If $G=G(m,1)$ and $A/G$ is smooth, then the pair $(A,G)$ corresponds to Example \ref{ex1};
\item If $G=G(3,3)$ and $A/G$ is smooth, then the pair $(A,G)$ corresponds to Example \ref{ex2};
\item If $G=G(4,2)$ and $A/G$ is smooth, then the pair $(A,G)$ corresponds to Example \ref{ex3};
\item If $G=G(6,p)$ with $p\geq 2$, then $A/G$ cannot be smooth.
\end{itemize}
In order to do this, we will construct a $G$-isogeny $B\to A$ such that the action of $G$ on $B$ is ``well-known''. Let us concentrate first in the cases where $m\neq p$. Then we obtain $B$ as follows:

Let $E_i$ be the image of $\mathbb{C}e_i$ in $A$ via the exponential map. We claim that it corresponds to an elliptic curve. Indeed, consider the non-trivial element $\tau=(\zeta_m^p,1)\in H$. Then a direct computation shows that $\im(1-\tau)=\C e_1$. This tells us that $E_1=(1-\tau)(A)$ and hence it corresponds to an elliptic curve. The same proof works for $E_2$.

Now, let $\Lambda_A$ be a lattice for $A$ in $\mathbb{C}^2$. Then $\mathbb{C}e_i\cap\Lambda_A$ corresponds to the lattice of $E_i$ in $\C=\C e_i$. We can thus define the $G$-stable sublattice of $\Lambda_A$
\[\Lambda_B:=(\C e_1\cap\Lambda_A)\oplus(\C e_2\cap\Lambda_A).\]
As in Section \ref{sec group actions}, this defines a $G$-isogeny $\pi:B\to A$. Moreover, we see that $B\simeq E_1\times E_2\simeq E^2$ and that $\pi|_{E_i}$ is an injection. Let $\Delta$ be the kernel of $\pi$. W will study the different possible quotients $A/G$ by studying the possible quotients $B/(\Delta\rtimes G)$ and thus by studying the possible $\Delta$'s. Our first result is the following:

\begin{lemma}\label{lemma Delta diagonal or hyper}
Assume that $m\neq p$. Then the coordinates of every element in $\Delta$ are invariant by $\zeta_m^p$, so in particular these elements are
\begin{itemize}
\item 2-torsion if $(m,p)\in\{(2,1),(4,1),(4,2),(6,3)\}$;
\item 3-torsion if $(m,p)\in\{(3,1),(6,2)\}$;
\item trivial if $(m,p)=(6,1)$.
\end{itemize}
\end{lemma}

\begin{proof}
Let $\bar t=(t_1,t_2)\in\Delta$. Then, since $\Delta$ is $G$-stable, we have that, for $\tau_1=(\zeta_m^p,1)\in H$,
\[(1-\tau_1)(\bar t)=((1-\zeta_m^p)t_1,0)\in\Delta.\]
But, by construction, there are no elements of the form $(x,0)$ in $\Delta$. We deduce then that $t_1$ is $\zeta_m^p$-invariant. The same proof works for $t_2$. The assertion on the torsion of $t_1$ and $t_2$ follows immediately.
\end{proof}

Let us study now pseudoreflections in $\Delta\rtimes G$. Define the elements
\begin{align*}
\rho &:=(\zeta_m,\zeta_m^{-1})\in H\subset G;\\
\sigma &:=(1\, 2)\in S_2\subset G;\\
\tau &:=(\zeta_m^p,1)\in H\subset G.
\end{align*}
Then there are two types of pseudoreflections in $G$:\label{pseudoref in Delta G}
\begin{itemize}
\item conjugates of $\rho^a\sigma$ for $0\leq a<\frac mp$;
\item conjugates of powers of $\tau$;
\end{itemize}
and the corresponding elliptic curves in $B$ are respectively:
\begin{align*}
E_{\rho^a\sigma}&=\{(x,-\zeta_m^ax)\mid x\in E\};\\
E_{\tau}&=\{(x,0)\mid x\in E\}.
\end{align*}
Recall that elements of the form $(x,0)$ are not in $\Delta$ by construction of the isogeny $\pi:B\to A$. Using Lemmas \ref{main lemma} and \ref{lemma Delta diagonal or hyper}, we obtain immediately the following result:

\begin{lemma}\label{lem pseudoref en DeltaG}
Every pseudoreflection in $\Delta\rtimes G$ that is not in $G$ is a conjugate of $(\bar t,\rho^a\sigma)$, where $0\leq a<\frac mp$, $\bar t=(t,-\zeta_m^at)\in\Delta$ and $t$ is $\zeta_m^p$-invariant.\qed
\end{lemma}

With these considerations, we can start a case by case study of the non-trivial $\Delta$'s. We recall that the main tool will be the Chevalley-Shephard-Todd Theorem, which states that $A/G=B/\Delta\rtimes G$ is smooth if and only if the stabilizer in $\Delta\rtimes G$ of each point in $B$ is generated by pseudoreflections.

\subsection{The case $G=G(2,1)$}
By Lemma \ref{lemma Delta diagonal or hyper}, we know that $\Delta$ is 2-torsion. Since we also know that there are no elements of the form $(t,0)$ for $t\in E$, we get the following possible options for $\Delta$:

\begin{enumerate}
\item $\Delta=\{0\}$;
\item $\Delta=\langle(t,t)\rangle$ with $t\in E[2]$;
\item $\Delta=\{(t,t)\mid t\in E[2]\}$;
\item $\Delta=\{(0,0),(t_1,t_2),(t_2,t_1),(t_1+t_2,t_1+t_2)\}$ with $t_1,t_2\in E[2]$, $t_1\neq t_2$.\\
\end{enumerate}

Case (1) clearly corresponds to Example \ref{ex1} (which gives a smooth quotient, cf.~\cite[Prop.~3.4]{ALA}). Case (2) cannot give a smooth quotient and this follows directly from \cite[Prop.~3.7]{ALA}.\footnote{The proof of this proposition only uses two variables and thus it works in dimension 2 as well.}\\

In case (3), we claim that the pair $(A,G)$ is isomorphic to the pair $(B,G)$. This will reduce us to the case with trivial $\Delta$, which was already dealt with. To prove the claim, consider the canonical basis of $T_0(A)=T_0(B)=\C^2$. Then the analytic representation of $G$ is given by the following values in its generators:
\[\rho_a((1,-1))=\begin{pmatrix} 1 & 0 \\ 0 & -1 \end{pmatrix},\quad  \rho_a((1\,2))=\begin{pmatrix} 0 & 1 \\ 1 & 0 \end{pmatrix}
\]
Now, with this basis and this $\Delta$, we can view the $G$-isogeny $B\to A$ as the morphism $E^2\to E^2$ given by the following matrix:
\begin{equation}\label{matrizG21}
M=\begin{pmatrix}\tag{*}
1 & 1 \\
1 & -1
\end{pmatrix},
\end{equation}
for which one can check that its kernel is precisely the elements in $\Delta$. In order to prove that the pairs $(A,G)$ and $(B,G)$ are isomorphic, it suffices thus to prove that the image of this representation of $G$ under conjugation by $M$ is $G$ once again. Direct computations give:
\begin{align*}
M\rho_a((1,-1))M^{-1}&=\begin{pmatrix} 0 & 1 \\ 1 & 0 \end{pmatrix}=\rho_a((1\,2)),\\
M\rho_a((1\,2))M^{-1}&=\begin{pmatrix} 1 & 0 \\ 0 & -1 \end{pmatrix}=\rho_a((1,-1)).
\end{align*}
And these clearly generate the same group $G$. \\

In case (4), consider the element $\bar t = (t_1',t_2')$ where $2t_i' = t_i$. Note that $G$ cannot fix $\bar t$ as $t_1'$ and $t_2'$ lie in different orbits by the action of $\mu_2$. Now, it is easy to see that there is no way the action of $\Delta$ can compensate the action of $G$ except in the case when we add the element $(t_1,t_2)$. A direct computation tells us then that the only element fixing $\bar t$ is $((t_1, t_2),(-1,-1)) \in \Delta \rtimes G$ and since this stabilizer is not generated by pseudoreflections by Lemma \ref{lem pseudoref en DeltaG}, we see that $A/G$ is not smooth.

\subsection{The case $G=G(4,2)$}\label{section G42}
Since $G(4,2)$ contains $G(2,1)$, we may start from the precedent list of possible non-trivial $\Delta$'s. However, these must also be stable by the new element $(i,i)\in H(4,2)$ (where $i=\zeta_4$). Defining by $t_0$ the only non-trivial $i$-invariant element in $E$, we get the following possibilities:

\begin{enumerate}
\item $\Delta=\{0\}$;
\item $\Delta=\langle(t_0,t_0)\rangle$;
\item $\Delta=\{(t,t)\mid t\in E[2]\}$;
\item $\Delta=\{(0,0),(t,t+t_0),(t+t_0,t),(t_0,t_0)\}$ with $t\in E[2]$, $t\neq t_0$.
\end{enumerate}

Case (1) does not give a smooth quotient $A/G$, cf.~\cite[Prop.~3.4]{ALA}. Case (2) corresponds to Example \ref{ex3} (and it actually gives a smooth quotient $A/G$ as we prove in section \ref{sec ex3}). Indeed, the $G$-isogeny $B\to A$ corresponds in this case to the morphism $E^2\to E^2$ with $E=\C/\Z[i]$ given by the matrix
\[\begin{pmatrix}
1 & -1 \\
0 & i-1
\end{pmatrix},\]
and the generators given in Example \ref{ex3} correspond to the conjugates by this matrix of the following matrices:
\[\left\{\begin{pmatrix} -1 & 0 \\ 0 & 1\end{pmatrix}\right.,\, \begin{pmatrix} -i & 0 \\ 0 & i\end{pmatrix},\, \left.\begin{pmatrix} 0 & 1 \\ 1 & 0\end{pmatrix} \right\}.\]
But these are clearly the matrix expressions of the generators $(-1,1),(-i,i)\in H$ and $(1\, 2)\in S_2$ of $G=H\rtimes S_2$.\\

In cases (3) and (4), we claim that the pair $(A,G)$ is isomorphic to the pair $(B,G)$. This will reduce us to the case with trivial $\Delta$, which was already dealt with. To prove the claim, we consider as for $G=G(2,1)$ the canonical basis of $T_0(A)=T_0(B)=\C^2$. Then the analytic representation of $G$ is given by the following values in its generators:
\[\rho_a((i,-i))=\begin{pmatrix} i & 0 \\ 0 & -i \end{pmatrix},\quad \rho_a((-1,1))=\begin{pmatrix} -1 & 0 \\ 0 & 1 \end{pmatrix},\quad  \rho_a((1\,2))=\begin{pmatrix} 0 & 1 \\ 1 & 0 \end{pmatrix}
\]
Now, with this basis and the $\Delta$ from case (2), we already know that $B\to A$ looks like $E^2\to E^2$ with matrix $M$ from \eqref{matrizG21}. It suffices to check then that the new generator $\rho_a((i,-i))$ falls into $\rho_a(G)$ after conjugation by $M$. And indeed we have that $M\rho_a((i,-i))M^{-1}=\rho_a((i,i))\rho_a((1\, 2))$.

With the $\Delta$ from case (3), the corresponding matrix for $B\to A$ is:
\[N=\begin{pmatrix}
1 & i \\
i & 1
\end{pmatrix}.\]
And once again, direct computations give:
\begin{align*}
N\rho_a((i,-i))N^{-1}&=\begin{pmatrix} 0 & -1 \\ 1 & 0 \end{pmatrix}=\rho_a((-1,1))\rho_a((1\, 2)),\\
N\rho_a((-1,1))N^{-1} &=\begin{pmatrix} 0 & -i \\ i & 0 \end{pmatrix}=\rho_a((1\,2))\rho_a((i,-i)),
\\
N\rho_a((1\,2))N^{-1}&=\begin{pmatrix} 0 & 1 \\ 1 & 0 \end{pmatrix}=\rho_a((1\,2)).
\end{align*}
And these clearly generate the same group $G$.

\subsection{The case $G=G(4,1)$}
Since $G(4,1)$ contains $G(4,2)$, we may start from the precedent list of possible non-trivial $\Delta$'s. Now, by Lemma \ref{lemma Delta diagonal or hyper}, we know that the coordinates of the elements in $\Delta$ are $i$-invariant. We get then that there are only two options for $\Delta$, that is the trivial case and $\Delta=\langle(t_0,t_0)\rangle$.

In the trivial case, we immediately see that $(A,G)$ corresponds to Example \ref{ex1}. Assume then that $\Delta$ is non-trivial and consider the element $(s,t)\in B$ with $s\in E[2]$, $s$ \emph{not} $i$-invariant and $2t=t_0$. Since clearly these elements have different order, we see that the orbits of these elements by the action of $\langle t_0\rangle\times\mu_4$ are different. Thus no action of an element in $\Delta\times H\subset\Delta\rtimes G$ can compensate the action of $(1\,2)\in G$ in order to fix $(s,t)$. In other words, the stabilizer of $\bar t$ must be contained in $\Delta\times H$. It is easy to see then that it corresponds to $\langle((t_0,t_0),(i,-1))\rangle$. By Lemma \ref{lem pseudoref en DeltaG}, this stabilizer is not generated by pseudoreflections and hence $A/G$ is not smooth in this case.

\subsection{The case $G=G(3,1)$}
By Lemma \ref{lemma Delta diagonal or hyper}, we know that the coordinates of the elements in $\Delta$ are $\zeta_3$-invariant. Now, there are only two such non-trivial elements that we will denote by $s_0$ and $-s_0$. Since we also know that there are no elements of the form $(t,0)$ for $t\in E$, we get the following possible options for a non-trivial $\Delta$:

\begin{enumerate}
\item $\Delta=\{0\}$;
\item $\Delta=\langle(s_0,s_0)\rangle$;
\item $\Delta=\langle(s_0,-s_0)\rangle$.
\end{enumerate}

We immediately see that the trivial case gives us Example \ref{ex1}. In case (2), Lemma \ref{lem pseudoref en DeltaG} tells us that the only pseudoreflections in $\Delta\rtimes G$ are those coming from $G$. In particular, in order to prove that $A/G$ cannot be smooth, it suffices to exhibit an element in $B$ such that its stabilizer in $\Delta\rtimes G$ has elements that are not in $G$. Now, by \cite[Lem.~2.8]{ALA} we know that there exists an element $\tau\in G$ such that $1-\tau$ is surjective. Then there exists an element $\bar z\in B$ such that $\bar z-\tau(\bar z)=(s_0,s_0)$. This implies that $((s_0,s_0),\tau)\in\Delta\rtimes G$ stabilizes $z$, proving thus that $A/G$ is not smooth in this case.\\

In case (3), consider the element $\bar s=(s,-s)\in B$ with $s\in E[3]$ and $s$ \emph{not} $\zeta_3$-invariant. Note that $\langle s_0\rangle\times\mu_3$ acts on $E[3]$ and a direct computation tells us that the orbit of $s$ is $\{s,s+s_0,s-s_0\}$. In particular, we see that $s$ and $-s$ lie in different orbits for this action. The same argument used in the case of $G(4,1)$ tells us then that the stabilizer of $\bar s$ must be contained in $\Delta\times H$. It is easy to see then that, up to changing $\bar s$ by $-\bar s$, it corresponds to $\langle((s_0,-s_0),(\zeta_3,\zeta_3))\rangle$. Since this stabilizer is not generated by pseudoreflections by Lemma \ref{lem pseudoref en DeltaG}, we see that $A/G$ is not smooth in this case as well.

\subsection{The case $G=G(6,1)$}
By Lemma \ref{lemma Delta diagonal or hyper}, we know that the only possibility is a trivial $\Delta$. This clearly corresponds to Example \ref{ex1}.

\subsection{The case $G=G(6,2)$}
Since $G(6,2)$ contains $G(3,1)$, we may start from the possible non-trivial $\Delta$'s for that case. Note that these are all 3-torsion subgroups. Thus, if $\bar x\in B$ denotes a 2-torsion element, we see that its stabilizer in $\Delta\rtimes G$ can only contain elements in $G$. Consider then the element $\bar t=(t,0)$ where $t$ is a non-trivial 2-torsion element. As it is proven in \cite[Prop.~3.4]{ALA}, the stabilizer of this element in $G$ is not generated by pseudoreflections. This implies that $A/G=B/(\Delta\rtimes G)$ cannot be smooth regardless of the choice of possible $\Delta$.

\subsection{The case $G=G(6,3)$}
Since $G(6,3)$ contains $G(2,1)$, we may start from the possible non-trivial $\Delta$'s for that case. Note that these are all 2-torsion subgroups. Thus, like we noticed in the previous case, if $\bar x\in B$ denotes a 3-torsion element, its stabilizer in $\Delta\rtimes G$ only contains elements in $G$. Consider then the element $\bar s=(s_0,0)$ where $s_0$ is a $\zeta_3$-invariant element (hence 3-torsion). Once again, as proven in \cite[Prop.~3.4]{ALA}, the stabilizer of this element in $G$ is not generated by pseudoreflections, which implies that $A/G$ cannot be smooth in any case of $\Delta$.\\

This finishes the study of the cases where $m\neq p$. We are left thus with the cases $G(3,3)$ and $G(6,6)$. In these particular cases we forget all the constructions done before and start from scratch.

\subsection{The case $G=G(3,3)$}
The group $G(3,3)$ is easily seen to be isomorphic \emph{as a complex reflection group} to $S_3$ acting on $\C^2$ via the standard representation. As such, it has already been treated by the first two authors in \cite[\S3.1]{ALA} and we know that in that case we get a smooth quotient if and only if we are in Example \ref{ex2}.

\subsection{The case $G=G(6,6)$}
Note that $G(6,6)$ is isomorphic to the direct product $G(3,3)\times\{\pm 1\}$. Since the actions of $S_3$ and $\mu_2=\{\pm 1\}$ commute, we may follow the approach taken by \cite{ALA} for $S_3$ and we will prove the following:

\begin{proposition}
Let $G(6,6)=S_3\times\mu_2$ act on an abelian surface $A$ in such a way that its action on $T_0(A)$ is the standard one for $S_3$ and the obvious one for $\mu_2$. Then $A/G$ is not smooth.
\end{proposition}

\begin{proof}
Let $\sigma=(1\, 2)\in S_3$ and $E=E_{\sigma}$ be induced by a line $L_\sigma\subseteq T_0(A)$, and define the lattice
\[\Lambda_B:=\sum_{\tau\in S_{3}}\tau(L_\sigma\cap\Lambda_A).\]
Since clearly all lattices are $\mu_2$-invariant, this gives us a $G$-invariant sublattice of $\Lambda_A$. Therefore, we get a $G$-equivariant isogeny $\pi:B\to A$ with kernel $\Delta$. Applying this construction to Example \ref{ex2}, to which we can naturally add the action of $\mu_2$ in order to get an action of $G$, we see that it gives the whole lattice. We can thus see $B$ as 
\[B=\{(x_1,x_2,x_3)\in E^3\mid x_1+x_2+x_3=0\},\]
where $S_3$ and $\mu_2$ act in their respective natural ways. Using the notations from Section \ref{sec group actions}, we see by inspection that $F_{\sigma,B}=E_{\sigma,B}[2]\simeq E[2]$, hence the map $\pi:F_{\sigma,B}\to F_{\sigma,A}$ is surjective since by Lemma \ref{lemma Esigma Dsigma}.2 we have $F_{\sigma,A}\subset E_{\sigma,A}[2]\simeq E[2]$. By Proposition \ref{prop Delta and D}, we have that $\Delta$ is contained in the fixed locus of all the conjugates of $\sigma$, which clearly generate $S_3$. Thus, $\Delta$ consists of elements of the form $(x,x,x)\in E^{3}$ such that $3x=0$. In particular, $\Delta$ is isomorphic to a subgroup of $E[3]$ and hence of order 1, 3 or 9.\\

Assume that $\Delta$ is trivial, that is, that $A=B$. Then the action of $G=S_3\times\mu_2$ on $B\simeq E^2$ induces an action of $\mu_2$ on $B/S_3\simeq\bb P^2$ (recall that the action of $S_3$ on $B$ is that of Example \ref{ex2}). We only need to notice then that any quotient of $\bb P^2$ by a non trivial action of the group $\mu_2$ is not smooth. This is well-known.\\

Assume now that $\Delta$ has order 3 and let $\bar t=(t,t,t)\in\Delta$ be a non-trivial element (thus $t\in E[3]$). Let $x\in E[3]$ be a non-trivial element different from $\pm t$ and consider $\bar x=(x,x+t,x-t)$. It is then easy to see that the element $(\bar t,(1\,2\,3))\in\Delta\rtimes G$ fixes $\bar x$ and that $\stab_G(\bar x)=\{1\}$, so that every pseudoreflection fixing $\bar x$ must lie outside $G$. Let $(\bar s,\sigma)$ be such a pseudoreflection. Using Lemma \ref{main lemma}, we see that $\sigma\in \{-(1\,2),-(2\,3),-(1\,3)\}$, where $-\tau$ denotes $(\tau,-1)\in S_3\times\mu_2=G$. Now, for any such $\sigma$, direct computations tell us that $(\bar s,\sigma)$ fixes $\bar x$ if and only $\bar s=(s,s,s)$ with $s=a_\sigma x+b_\sigma t$ for some $a_\sigma\neq 0$. Since $x\not\in\langle t\rangle\subset E[3]$, we see that $\bar s\not\in\Delta$ and hence these pseudoreflections do not exist. We get then that $\stab_{\Delta\rtimes G}(\bar x)$ is not generated by pseudoreflections and hence $A/G$ cannot be smooth.\\

Assume finally that $\Delta$ has order 9. We claim that in this case the pair $(A,G)$ is isomorphic to the pair $(B,G)$. This will reduce us to the case with trivial $\Delta$, which was already dealt with. To prove the claim, fix the basis $\{(1,0,-1),(0,1,-1)\}$ of $T_0(B)=T_0(A)\subset\C^3$. Then the analytic representation of $G$ is given by the following values in its generators:
\[\rho_a((1\,2))=\begin{pmatrix} 0 & 1 \\ 1 & 0 \end{pmatrix},\qquad \rho_a(-1)=\begin{pmatrix} -1 & 0 \\ 0 & -1 \end{pmatrix},\qquad  \rho_a((1\,2\, 3))=\begin{pmatrix} -1 & -1 \\ 1 & 0 \end{pmatrix}
\]
Now, with this basis and this $\Delta$, the analytic representation of $B\to A$ is given by the inverse of the following matrix:
\[M=\begin{pmatrix}
-1 & -2 \\
2 & 1
\end{pmatrix}.\]
Indeed, this corresponds to the morphism that sends $(x,y,-x-y)\in B\subset E^3$ to $(-x-2y,2x+y,-x+y)\in A\subset E^3$ and thus its kernel is precisely the elements of the form $(x,x,x)\in E[3]^3\subset B$, that is, $\Delta$. In order to prove that the pairs $(A,G)$ and $(B,G)$ are isomorphic, it suffices thus to prove that the image of this representation of $G$ under conjugation by $M$ is $G$ once again. Direct computations give:
\begin{gather*}
M\rho_a(-1)M^{-1}=\rho_a(-1),\qquad M\rho_a((1\,2\, 3))M^{-1}=\rho_a((1\,2\,3)),\\
M\rho_a((1\,2))M^{-1}=\begin{pmatrix} 0 & -1 \\ -1 & 0 \end{pmatrix}=\rho_a((1\,2))\rho_a(-1).
\end{gather*}
And these clearly generate the same group $G$.
\end{proof}

\section{Proof of $(3)\Rightarrow (2)$}\label{sec ex3}
The only case left to prove in this last implication is Example \ref{ex3} (the other two are proved in \cite{Auff}). Let us study then this example in detail.\\

Recall that in section \ref{section G42} we proved that the pair $(A,G)$ from Example \ref{ex3} can be obtained as follows. Let $G=G(4,2)$ and let $B=E^2$ with $E=\C^2/\Z[i]$. Denote by $t_0$ the $i$-invariant element in $E$ and denote by $q_0$ the quotient morphism $E\to E/\langle t_0\rangle\simeq E$. Then $A=B/\Delta$ with $\Delta=\langle(t_0,t_0)\rangle\in E^2=B$ and the action of $G$ on $A$ is the one induced by $B\to A$.

Note now that $G$ has an index 2 subgroup $G_1:=G(2,1)=H_1\rtimes S_2$, which is thus normal in $G$ (here, $H_1=\{\pm 1\}^2$). Moreover, the pair $(B,G_1)$ corresponds to Example \ref{ex1}, so that $B/G_1\simeq\bb P^2$. Finally, note that $\Delta$ is an order 2 subgroup of $B$ and thus $G$ acts \emph{trivially} on it. In particular, the actions of $G$ and $\Delta$ on $B$ commute and hence we have a commutative diagram of Galois covers
\[\xymatrix{
B \ar[r]^{\Delta} \ar[d]^{G_1} \ar@/_1.5pc/[dd]_{G} & A \ar[d] \ar@/^1.5pc/[dd]^{G} \\
\bb P^2 \ar[r] \ar[d]^{G/G_1} & A/G_1 \ar[d] \\
B/G \ar[r] & A/G,
}\]
where parallel arrows have the same Galois group. Since $\Delta$ and $G/G_1$ have both order 2, we see then that $A/G$ is a quotient of $\bb P^2$ by the action of a Klein group.

\begin{proposition}\label{prop ex3 smooth}
The quotient $A/G$ is isomorphic to $\mathbb{P}^2$.
\end{proposition}

This proposition finishes the proof of $(3)\Rightarrow(2)$ in Theorem \ref{classification}.

\begin{rem}
This example was already known to Tokunaga and Yoshida (cf.~\cite[\S5, Table II]{TY}). However, in order to prove that $A/G\simeq\bb P^2$, they cite an article by Shvartsman which contains no proofs (cf.~\cite{Shvartsman}).
\end{rem}

\begin{proof}
Since $A/G$ is a quotient of $\bb P^2$ by the action of a Klein group $K$, the only thing we need to check is that this action gives $\bb P^2$ as a quotient. Note first that the action is faithful since it comes from the faithful action of $G\times\Delta$ on $B$. Consider then $K$ as a subgroup of $\mathrm{PGL}_3=\aut(\bb P^2)$ and let $K_1$ be its preimage in $\mathrm{SL}_3$. This is an order 12 group and hence any $2$-Sylow subgroup of $K_1$ gives a lift of $K$ to a subgroup of $\mathrm{GL}_3$. This implies that the action lifts to $\C^3$ and it can thus be seen as a linear representation of $K$. Since there are exactly four irreducible representations of $K$, all of dimension 1, a direct check tells us that any choice of three different representations gives the same faithful action on $\bb P^2$ up to conjugation, whereas any other choice gives a non-faithful action. We may assume then that the nontrivial elements $x_i\in K$ for $i=1,2,3$ act on $\bb P^2$ via the diagonal matrices with $1$ on the $i$-th coordinate and $-1$ elsewhere. The quotient of $\bb P^2$ by such a group is the weighted projective space $\bb P(2,2,2)$, which is well-known to be isomorphic to $\bb P(1,1,1)=\bb P^2$. This concludes the proof.
\end{proof}

\end{document}